\documentclass[10pt]{amsart}
\usepackage{amssymb}
\usepackage{dsfont}
\usepackage{amscd}
\usepackage[mathscr]{euscript} 
\usepackage[all]{xy}
\usepackage{url}
\usepackage{stmaryrd}
\usepackage{comment}
\usepackage[retainorgcmds]{IEEEtrantools}
\usepackage{hyperref}
\usepackage{graphicx}
\usepackage{mathtools}
\usepackage{color}
\usepackage{faktor}

\title{Existentially closed fields with finite group actions}
\author[D.M. HOFFMANN]{Daniel Max Hoffmann$^{\dagger}$}
\thanks{2010 \textit{Mathematics Subject Classification}. Primary 03C60; Secondary 03C45, 12H10}
\thanks{\textit{Key words and phrases}. difference fields.}
\thanks{$^{\dagger}$SDG. Part of this work was conducted during first author's internship at the Warsaw Center of Mathematics and Computer Science.
Supported by Narodowe Centrum Nauki grants no. 2015/19/B/ST1/01150, 2016/20/T/ST1/00482 and 2016/21/N/ST1/01465.}
\address{$^{\dagger}$Instytut Matematyczny\\
Uniwersytet Wroc{\l}awski\\
Wroc{\l}aw\\
Poland}
\email{daniel.hoffmann@math.uni.wroc.pl}
\urladdr{https://www.researchgate.net/profile/Daniel\_Hoffmann8}

\author[P. KOWALSKI]{Piotr Kowalski$^{\spadesuit}$}
\thanks{$^{\spadesuit}$Supported by T\"{u}bitak grant 2221 and supported by Narodowe Centrum Nauki grants no. 2012/07/B/ST1/03513,
2015/19/B/ST1/01150 and 2015/19/B/ST1/01151.}
\address{$^{\spadesuit}$Instytut Matematyczny\\
Uniwersytet Wroc{\l}awski\\
Wroc{\l}aw\\
Poland}
\email{pkowa@math.uni.wroc.pl} \urladdr{http://www.math.uni.wroc.pl/\textasciitilde pkowa/ }

\DeclareMathOperator{\su}{SU}

\DeclareMathOperator{\acl}{acl}  
 \DeclareMathOperator{\aut}{Aut} \DeclareMathOperator{\id}{id}
 \DeclareMathOperator{\fr}{Fr} \DeclareMathOperator{\Fr}{Fr}

  \DeclareMathOperator{\gal}{Gal}
\DeclareMathOperator{\ch}{char}  
 
 \DeclareMathOperator{\theo}{Th}\DeclareMathOperator{\alg}{alg}

\DeclareMathOperator{\tp}{tp}

\DeclareMathOperator{\spec}{Spec}

\DeclareMathOperator{\ddf}{DF}\DeclareMathOperator{\dcf}{DCF}\DeclareMathOperator{\scf}{SCF}

\DeclareMathOperator{\tcf}{TCF}

\DeclareMathOperator{\Func}{Func}

\DeclareMathOperator{\perf}{perf}
\DeclareMathOperator{\rk}{rk}

\newtheorem{theorem}{Theorem}[section]
\newtheorem{prop}[theorem]{Proposition}
\newtheorem{lemma}[theorem]{Lemma}
\newtheorem{cor}[theorem]{Corollary}

\theoremstyle{definition}
\newtheorem{definition}[theorem]{Definition}
\newtheorem{example}[theorem]{Example}
\newtheorem{remark}[theorem]{Remark}
\newtheorem{question}[theorem]{Question}

\theoremstyle{remark}

\makeatletter
\providecommand*{\cupdot}{%
  \mathbin{%
    \mathpalette\@cupdot{}%
  }%
}
\newcommand*{\@cupdot}[2]{%
  \ooalign{%
    $\m@th#1\cup$\cr
    \hidewidth$\m@th#1\cdot$\hidewidth
  }%
}
\makeatother

\def\Ind#1#2{#1\setbox0=\hbox{$#1x$}\kern\wd0\hbox to 0pt{\hss$#1\mid$\hss}
\lower.9\ht0\hbox to 0pt{\hss$#1\smile$\hss}\kern\wd0}

\def\notind#1#2{#1\setbox0=\hbox{$#1x$}\kern\wd0
\hbox to 0pt{\mathchardef\nn=12854\hss$#1\nn$\kern1.4\wd0\hss}
\hbox to 0pt{\hss$#1\mid$\hss}\lower.9\ht0 \hbox to 0pt{\hss$#1\smile$\hss}\kern\wd0}

\begin{document}

\newcommand{\twoc}[3]{ {#1} \choose {{#2}|{#3}}}
\newcommand{\thrc}[4]{ {#1} \choose {{#2}|{#3}|{#4}}}
\newcommand{\Zz}{{\mathds{Z}}}
\newcommand{\Ff}{{\mathds{F}}}
\newcommand{\Cc}{{\mathds{C}}}
\newcommand{\Rr}{{\mathds{R}}}
\newcommand{\Nn}{{\mathds{N}}}
\newcommand{\Qq}{{\mathds{Q}}}
\newcommand{\Kk}{{\mathds{K}}}
\newcommand{\Pp}{{\mathds{P}}}
\newcommand{\ddd}{\mathrm{d}}
\newcommand{\Aa}{\mathds{A}}
\newcommand{\dlog}{\mathrm{ld}}
\newcommand{\ga}{\mathbb{G}_{\rm{a}}}
\newcommand{\gm}{\mathbb{G}_{\rm{m}}}
\newcommand{\gaf}{\widehat{\mathbb{G}}_{\rm{a}}}
\newcommand{\gmf}{\widehat{\mathbb{G}}_{\rm{m}}}
\newcommand{\gdf}{\mathfrak{g}-\ddf}
\newcommand{\gdcf}{\mathfrak{g}-\dcf}
\newcommand{\fdf}{F-\ddf}
\newcommand{\fdcf}{F-\dcf}
\newcommand{\mw}{\scf_{\text{MW},e}}

\maketitle
\begin{abstract}
We study algebraic and model-theoretic properties of existentially closed fields with an action of a fixed finite group. Such fields turn out to be pseudo-algebraically closed in a rather strong sense. We place this work in a more general context of the model theory of fields with a (finite) group scheme action. 
\end{abstract}

\section{Introduction}
A difference field (or a transformal field) is a field with a distinguished collection of endomorphisms. The model theory of difference fields was initiated in \cite{macin1} and analyzed in depth in \cite{acfa1} and \cite{acfa2}. In these works, universal (or generic) difference fields were considered. In terms of group actions, existentially
closed $\mathbb{Z}$-actions on fields were studied in these papers.
The model theory of actions of different types of groups has been also studied. We list several examples below.
\begin{itemize}
\item Actions of the free groups $F_n$: Hrushovski \cite{Hr2}, Kikyo-Pillay \cite{KiPi}, Sj\"{o}gren \cite[Theorem 16]{sjogren}, Moosa-Scanlon
\cite[Proposition 4.12]{moosca2}; model companion (called ACFA$_n$) exists.

\item Actions of $\mathbb{Q}$: Medvedev \cite{med1}; model companion (called $\mathbb{Q}$ACFA) exists.

\item Actions of $\Zz\times \Zz$: Hrushovski (unpublished, 
see also Theorem 3.2 in \cite{kikyo2}); no model companion.
\end{itemize}
In this paper, we deal with the model theory of actions of a finite group by field automorphisms.

After writing a preliminary version of this paper, we learned (to our great surprise) from Zo\'{e} Chatzidakis that a similar topic was already picked up by Nils Sj\"{o}gren in \cite{sjogren}, where he considered model theory 
of actions of an \emph{arbitrary} group on fields. The reference \cite{sjogren} is a ``research report'' which consists of Sj\"{o}gren's PhD Thesis (Filosofie licentiatavhandling). We are not sure whether \cite{sjogren} is the final version of the thesis. As far as we know, Sj\"{o}gren's work has never been published in a peer-reviewed journal. It turned out that our work here, although entirely independent, has quite a large intersection with \cite{sjogren}. In the cases of such an intersection, we always quote \cite{sjogren} and also try to comment on Sj\"{o}gren's version of a given argument. There is one case where Sj\"{o}gren had obtained a full result (Theorem \ref{sjor}) and we were able to analyze a special case only (Example \ref{ex39}); we state his result in Section \ref{secufc} and provide a quick argument there.

We study Galois field extensions with a fixed finite Galois group $G$. We call such an extension \emph{$G$-transformal field}. We give geometric axioms of the theory of existentially closed $G$-transformal fields and call the resulting theory $G-\tcf$. Using these axioms, we show that the underlying field and the field of constants of any model of $G-\tcf$ are pseudo-algebraically closed (abbreviated PAC). We describe purely algebraically the constant fields of models of $G-\tcf$: they are perfect PAC fields satisfying an extra ``$G$-closedness'' condition (see Theorem \ref{equiv}). The $G$-closedness condition implies that the underlying PAC field is bounded (see Corollary \ref{4_to_2}), hence applying known results about PAC fields we conclude that the theory $G-\tcf$ is supersimple of finite SU-rank (Proposition \ref{su_su}). We also show that the theory of the PAC fields we consider is model complete only after adding finitely many constant symbols (Theorem \ref{complete411}), as opposed to the general PAC case, where infinitely many constants are needed (see \cite[Prop. 4.6]{ChaPil}).

We characterize the class of PAC fields which appear as constant fields of models of $G-\tcf$. It turns out that such PAC fields are not ``very far'' from being algebraically closed, they satisfy the ``$K$-strongly PAC'' condition (Corollary \ref{4_to_2}). This condition says that if $K$ is the underlying field of a model of $G-\tcf$ and $C$ is the field of constants, then any $K$-irreducible $C$-variety has a $C$-rational point.
Therefore, this condition is much stronger than PAC since we need the irreducibility only after extending the scalars to a finite field extension.
Testing the existence of rational points on the $K$-irreducible varieties (a much wider class than the absolutely irreducible ones) is related to a similar condition given in the axioms for $G-\tcf$ (Definition \ref{def_axioms}).
Since our perfect PAC fields are not quasi-finite (i.e. the absolute Galois group is not isomorphic to $\widehat{\Zz}$), they are not pseudo-finite. A full description of the absolute Galois group of the constant field of a model of $G-\tcf$ was obtained by Sj\"{o}gren in \cite{sjogren} (it is \emph{universal Frattini cover} of $G$, see \cite[Section 22.6]{FrJa}), we quote Sj\"{o}gren's results in Theorem \ref{sjor}.

In a broader context, this work should be seen as a part of our project of investigating the model theory of finite group scheme actions. To understand the content of this paper, it is not necessary to know what a (finite) group scheme is. However, to put our work to a wider perspective, in this paragraph, we freely use the notion of (finite) group schemes and some of their properties (a reader may consult \cite{Water} for necessary definitions). Let $\mathfrak{g}$ be a finite group scheme. Then the connected component $\mathfrak{g}^0$ of $\mathfrak{g}$ is an infinitesimal group scheme and the quotient $\mathfrak{g}/\mathfrak{g}^0$ is an \'{e}tale finite group scheme.  Over an algebraically closed field, an \'{e}tale finite group scheme may be identified with a finite group $G$ (over an arbitrary field we need to take into account an action of the absolute Galois group on $G$). The actions of infinitesimal group schemes correspond to truncated Hasse-Schmidt derivations and their model theory was worked out in \cite{HK}. This paper deals with the model theory of the actions of the discrete quotient $G$. In a subsequent work, we plan to analyze the more general case of the model theory of actions of arbitrary finite group schemes. This research fits to even a more general circle of topics (see e.g. \cite{kamen1} and \cite{moosca2}) concerning the model theory of fields with operators.

The paper is organized as follows. In Section \ref{sec2}, we show basic facts about $G$-transformal fields and give a geometric axiomatization for the class of existentially closed $G$-transformal fields (obtaining the theory $G-\tcf$). In Section \ref{sec3}, we prove algebraic properties of existentially closed $G$-transformal fields and characterize them as $G$-closed perfect pseudo-algebraically closed fields. In Section \ref{logic}, we use the results of Section \ref{sec3} to obtain model-theoretical properties of the theory $G-\tcf$. In Section \ref{last}, we describe more precisely how the results of this paper fit into more general contexts (arbitrary affine group scheme or a possibly infinite group) and we state several questions.

We would like to thank the referee for correcting several errors and for many interesting mathematical comments.

\section{The theory of $G$-transformal fields and its model companion}\label{sec2}
Let $G=\lbrace g_1,\ldots,g_{e}\rbrace$ be a finite group of order $e$, with the neutral element $g_1=1$. 
By $\mathcal{L}_G=\lbrace +,\cdot,\sigma_{1},\ldots,\sigma_{e},0,1\rbrace$ we denote the language of rings with $e$ additional unary operators.
Sometimes $\sigma_k$ will be denoted by $\sigma_{g_k}$ or $\sigma_g$, and
$(\sigma_1,\ldots,\sigma_e)$ by $\overline{\sigma}$ or even by $G$.
For any $\mathcal{L}_G$-structure $(R,\overline{\sigma})$ and any $\bar{r}=(r_1,\ldots,r_n)\in R^n$ we use the following convention
$$\sigma_k(\bar{r})=\big(\sigma_k(r_1),\ldots,\sigma_k(r_n)\big),$$
where $k\leqslant e$. If $\bar{r}_1,\ldots,\bar{r}_e\in R^n$, then
$$\bar{\sigma}\big(\bar{r}_1,\ldots,\bar{r}_e\big):=
\big(\sigma_1(\bar{r}_1),\ldots,\sigma_e(\bar{r}_e)\big),$$
$$\bar{\sigma}(\bar{r}_1):=\big(\sigma_1(\bar{r}_1),\ldots,\sigma_e(\bar{r}_1)\big).$$
\begin{definition}
 An $\mathcal{L}_G$-structure $(R,\overline{\sigma})$ is called $G$-\emph{transformal ring} if
 \begin{itemize}
  \item[i)] the structure $(R,+,\cdot,0,1)$ is a ring,
  \item[ii)] the map $G\ni g_k\mapsto\sigma_k\in\aut(R)$ is a group homomorphism.
 \end{itemize}
 An $\mathcal{L}_G$-structure $(K,\overline{\sigma})$ is called $G$-\emph{transformal field} if $(K,\overline{\sigma})$ is a $G$-transformal ring and $K$ is a field. 
 We will consider $G$-\emph{transformal extensions}, i.e. ring extenions
preserving the whole structure of $G$-transformal rings.
\end{definition}

\begin{definition}
Assume that $(R,\overline{\sigma})$ is a $G$-transformal ring. We consider the set
$$R^G=\lbrace a\in R\;\;|\;\;\sigma_k(a)=a\text{ for each }k\leqslant e\rbrace$$
and call it \emph{ring of constants} or \emph{field of constants} if $R$ is a field.
A $G$-transformal field $(K,\overline{\sigma})$ is called \emph{strict} if $[K:K^G]=e \;(=|G|)$.
\end{definition}
\noindent
Our choice of the name ``strict" is related with the 
analogous name in the case of Hasse-Schmidt derivations (\cite[Definition 3.5]{HK}, \cite[page 1.]{Zieg3}), 
where, together with some assumptions about the degree of imperfection, it leads to linear disjointness of constants in Hasse-Schmidt extensions.
In Remark \ref{rem_Cspace}(1) below, we  will see that strict $G$-transformal fields enjoy a similar property.

\begin{remark}\label{rem_Cspace}
Let $(K,\overline{\sigma})\subseteq(K',\overline{\sigma}')$ be a $G$-transformal extension, $(K,\overline{\sigma})$ be a strict $G$-transformal field, $C:=K^G$ and $C':=(K')^G$. Moreover, we 
fix $\lbrace v_1,\ldots,v_{e}\rbrace$, a $C$-linear basis of $K$.
\begin{enumerate}
\item Clearly, we have
$$[K':C']=|\gal(K'/C')|\leqslant e.$$
On the other hand, the map $G\to \aut(C'(v_1,\ldots,v_e))$ is an embedding, so we have
$$[C'(v_1,\ldots,v_e):C']\geqslant e.$$
Therefore $K'=C'(v_1,\ldots,v_e)$ and the elements $v_1,\ldots,v_e$ are linearly independent over $C'$. Hence $K$ and $C'$ are linearly disjoint over $C$.

\item Because $[K:C]=e$, there exists a finite tuple $\bar{c}\subseteq C$ which defines a multiplication
on the $C$-linear space $C^{e}$, which coincides with the multiplication on the field $K$. More precisely, we determine $\bar{c}=(c_{i,j,l})_{i,j,l\leqslant e}$ by
$$v_i\cdot v_j=\sum\limits_{l=1}^{e} c_{i,j,l}\,v_l,$$
where $v_i\cdot v_j$ denotes the multiplication in $K$. In a similar way, we choose parameters $\bar{d}=(d_{k,j,l})_{k,j,l\leqslant e}$ which define the
action of $G$ by automorphisms, i.e.
$$\sigma_k(v_j)=\sum\limits_{l=1}^{e} d_{k,j,l}\,v_l.$$
 
\item By item (1), the above discussed constants $\bar{c}'$ and $\bar{d}'$ 
 for $(C',K')$ can be choosen inside the field $C$.
 
\item Choose $\bar{c},\bar{d}\subseteq C$ right for both 
$(C,K)$ and $(C',K')$ as in item (3). On $(C')^e$ we introduce a $G$-transformal field structure, i.e. the addition and the multiplication ($e^2$-ary functions on $C'$) are given by
 \begin{equation*}
   \begin{bmatrix}
   \alpha_1\\
   \vdots\\
   \alpha_e
  \end{bmatrix}\oplus
  \begin{bmatrix}
   \beta_1\\
   \vdots\\
   \beta_e
  \end{bmatrix} =
 \begin{bmatrix}
  \alpha_1+\beta_1\\
  \vdots \\
  \alpha_e+\beta_e
 \end{bmatrix},\quad 
  \begin{bmatrix}
   \alpha_1\\
   \vdots\\
   \alpha_e
  \end{bmatrix}\odot
  \begin{bmatrix}
   \beta_1\\
   \vdots\\
   \beta_e
  \end{bmatrix} =
 \begin{bmatrix}
  \sum\limits_{i,j\leqslant e}\alpha_i\beta_j c_{i,j,1}\\
  \vdots \\
  \sum\limits_{i,j\leqslant e}\alpha_i\beta_j c_{i,j,e}
 \end{bmatrix},
\end{equation*}
and for each $k\leqslant e$ the automorphism (an $e$-ary function on $C'$) is given by
\begin{equation*}
 \tilde{\sigma}_k\begin{bmatrix}
  \alpha_1\\
  \vdots\\
  \alpha_e
 \end{bmatrix} = 
 \begin{bmatrix}
  \sum\limits_{j\leqslant e}d_{k,j,1}\,\alpha_j\\
  \vdots\\
  \sum\limits_{j\leqslant e}d_{k,j,e}\,\alpha_j
 \end{bmatrix}
\end{equation*}
(the constants $0$ and $1$ are of the form $(0,\ldots,0)$ and $(1,0,\ldots,0)$).
Then the function 
$$f:K'\ni \alpha_1 v_1+\ldots+\alpha_e v_e\mapsto (\alpha_1,\ldots,\alpha_e)\in (C')^e$$
is an isomorhism between $G$-transformal fields.
Due to the choice of $\bar{c}$ and $\bar{d}$, we obtain that $\oplus$, $\odot$ and $(\tilde{\sigma}_k)_{k\leqslant e}$ restrict to $C^e$ and $f|_{K}:K\to C^e$ is an isomorphism of $G$-transformal fields as well. It is obvious, but important to note that the following diagram of $G$-transformal
morphisms commutes
\begin{equation*}\label{easy_diag}
 \xymatrix{  K \ar[r]^{\cong}_{f|_K}\ar[d]_{\subseteq} & C^e \ar[d]^{\subseteq} \\
 K' \ar[r]^{\cong}_{f} & (C')^e.}
\end{equation*}
\end{enumerate}
\end{remark}

\begin{definition}
Let $(R,\overline{\sigma})$ be a $G$-transformal ring and let $I$ be an ideal of $R$.
We say that $I$ is $G$-\emph{invariant} ideal if for each $k\leqslant e$, we have
$$\sigma_k(I)\subseteq I.$$
\end{definition}\noindent
\begin{remark}\label{linijka1}
Note that the above condition is equivalent to $\sigma_k(I)=I$ for each $k\leqslant e$. 
Clearly, for a $G$-invariant ideal $I$, there is a unique $G$-transformal structure on $R/I$ such that the quotient map is a $G$-transformal map.
\end{remark}

We define a group action of $G$ on $\lbrace 1,\ldots,e\rbrace$ as follows:
$$g_k\ast l=j \iff g_kg_l=g_j,$$
so $g_1$ still acts as the identity.
For a fixed $n>0$ and $i\leqslant e$, let $\pmb{X}_i$ denote the $n$-tuple of variables, $(X_{i,1},\ldots,X_{i,n})$. 
For a $G$-transformal ring $(R,\overline{\sigma})$, we introduce a $G$-transformal
ring structure on the ring $R[\pmb{X}_1,\ldots,\pmb{X}_e]$, by
$$\sigma_{g_k}(f(\pmb{X}_1,\ldots,\pmb{X}_e))=f^{\sigma_{g_k}}(\pmb{X}_{g_k\ast 1},\ldots,\pmb{X}_{g_k\ast e}),$$
where 
$$\Big(\sum\limits_{\mathbf{i}} r_{\mathbf{i}} \bar{\pmb{X}}^{\mathbf{i}} \Big)^{\sigma_{g_k}}=
\sum\limits_{\mathbf{i}} \sigma_{g_k}(r_{\mathbf{i}}) \bar{\pmb{X}}^{\mathbf{i}} .$$
For any ring $R$ and any ideal $I\trianglelefteqslant R[\pmb{X}_1,\ldots,\pmb{X}_e]$, we introduce
$$V_R(I):=\lbrace \bar{r}\in R^{ne}\;|\;  (\forall f\in I)(f(\bar{r})=0)\rbrace.$$

\begin{definition}\label{def_axioms}
 A $G$-transformal field $(K,\overline{\sigma})$ is a model of $G-\tcf$, if
 for every $n\in\mathbb{N}_{>0}$ and $|\pmb{X}_1|=\ldots=|\pmb{X}_e|=n$, it satisfies the following axiom scheme:
\begin{itemize}
 \item[($\clubsuit$)]  
 for any $I,J\trianglelefteqslant K[\pmb{X}_1,\ldots,\pmb{X}_e]$ such that $I\subsetneq J$
 and $I$ is a $G$-invariant prime ideal,
 there is $a\in K^n$ satisfying $\overline{\sigma}(a)\in V_K(I)\setminus V_K(J)$.
\end{itemize}
\end{definition}

\begin{remark}\label{rem_axioms}
We will see now that the above axiom scheme is actually first order. We give a geometric version of it and discuss how does it correspond to the geometric axioms of ACFA and of some other theories.
\begin{enumerate}
 \item 
 It is rather standard to see that the above scheme of axioms are first-order, but we still give a detailed argument using \cite{bounds}.
\begin{enumerate}
\item
By \cite[Theorem 2.10.i)]{bounds}, for each $n,d\in\mathbb{N}$ there exists     $B(n,d)\in\mathbb{N}$ such that for every field $K$ and
 every $I\trianglelefteqslant K[X_1,\ldots,X_n]$, generated by polynomials of degree $\leqslant d$, the following are equivalent
 \begin{itemize}
 \item $I\text{ is prime or }1\in I$
 \item $\text{for all }f,g\in K[X_1,\ldots,X_n]\text{ of degree }\leqslant B(n,d),\text{ if } fg\in I,\\\text{ then } f\in I\text{ or }g\in I.$
 \end{itemize}
  
 Moreover, the Property (I) at \cite[p. 78]{bounds} says that for $f_0,\ldots,f_m\in K[X_1,\ldots,X_n]$, all of degree $\leqslant d$, we have
 $$f_0\in (f_1,\ldots,f_m)\quad\Rightarrow\quad(\exists h_1,\ldots,h_m)\big(f_0=\sum\limits_{i=1}^{m}h_if_i\big),$$
 where the degree of each $h_1,\ldots, h_m$ is bounded by a universal constant $A(n,d)$. 
 Therefore being a prime ideal in $K[X_1,\ldots,X_n]$ is a first order statement over a field $K$, coded by a tuple $(a_{i_1,\ldots,i_n,s})\subseteq K$,
 where $i_1,\ldots,i_n\leqslant d$ and $s\leqslant m$, corresponding to generators of this prime ideal in $K[X_1,\ldots,X_n]$.
 The case $1\in I$ can not occur, since in our case there is an ideal $J$ properly containing $I$.
 
\item 
 The condition $I\subsetneq J$ can be expressed as follows
 $$f_1,\ldots,f_M\in J,$$
 $$g_1\not\in I\;\vee\ldots\vee\;g_N\not\in I,$$
 where $I=(f_1,\ldots,f_M)$, $J=(g_1,\ldots,g_N)$ and
 the symbols ``$\in$'' and ``$\not\in$'' are considered in $K[X_1,\ldots,X_{e}]$.
 By the above discussion, this is a first order statement over the field $K$.
 
\item
By the Property (I) from \cite[p. 78]{bounds}, the assumption ``$I$ is $G$-invariant" is also a first-order statement.
For $I$ generated by a polynomial in one variable $f$, the $G$-invariantness of $I$ assures us that permuting roots (by the action of $G$) of $f$ in its splitting field extension will lead to a $G$-transformal structure on this extension.
\end{enumerate} 
  \item 
We will give a geometric interpretation of $(\clubsuit)$, i.e. we express $(\clubsuit)$ in the language of schemes. The language of schemes (rather than of algebraic varieties) is necessary here, since our schemes need not be absolutely irreducible. We need to provide some set-up first.
\begin{enumerate}
\item The group $G$ acts on the set $\{1,\ldots,e\}$ 
(as in the paragraph after Remark \ref{linijka1}) and this action gives a coordinate action on the affine space $\Aa^{ne}$ by automorphisms.
    
\item Let $T$ be a $K$-scheme and $\sigma\in \aut(K)$. Then one defines a ``twisted'' $K$-scheme as follows
$$T^{\sigma}:=T\times_{\spec(K)}(\spec(K),\spec(\sigma)).$$
If $T=\spec(K[X_1,\ldots,X_e]/I)\subseteq \Aa^{ne}$, then $T^{\sigma}$ may be identified with $\spec(K[X_1,\ldots,X_e]/I^{\sigma})\subseteq \Aa^{ne}$, where $I^{\sigma}:=\{f^{\sigma}\ |\ f\in I\}$.
\end{enumerate}
The geometric version of $(\clubsuit)$ can be expressed as follows.
\begin{itemize}
\vspace{2mm}
\item[$(\clubsuit^{\mathrm{g}})$] Suppose that $V\subseteq \Aa^{ne}$ is a $K$-irreducible $K$-subscheme and $W\varsubsetneq V$ is a proper $K$-subscheme. If for any $\sigma\in G$, we have $\sigma(V)=V^{\sigma}$ (where $\sigma(V)$ is as in (a) above and $V^{\sigma}$ is as in (b) above), then there is $a\in \Aa^n(K)$ such that $(\sigma_1(a),\ldots,\sigma_e(a))\in V(K)\setminus W(K)$.
\vspace{2mm}
\end{itemize}
 
 \item At first, it looks like our axioms differ in their form from the axioms of ACFA in three ways.
\begin{enumerate}
\item We need to consider only the ambient varieties $U$ of the form $\Aa^n$.

\item We do not need the projection maps to be dominant.

\item We have an extra ``$G$-invariance'' assumption.
\end{enumerate}
However, it turns out that the ``$G$-invariance'' assumption implies that our axioms are closer to the axioms of ACFA, than they seem to be (similarly as for derivations of Frobenius \cite{K2} or Hasse-Schmidt derivations \cite{K3}, \cite{HK1}). Assume for simplicity that $|G|=2$, so a $G$-difference field has two automorphisms: $\id,\sigma$. If $V_i$ is the Zariski closure of the projection of $V$ on the $i$-th coordinate ($i=1,2$), then $\sigma(V_1)=V_2$ (and $\sigma(V_2)=V_1$). Hence $V\subseteq V_1\times \sigma(V_1)$ and the projections $V\to V_1,V\to \sigma(V_1)$ \emph{are} dominant as in the case of axioms of ACFA. Similarly for an arbitrary finite group $G$.

  \item There are some similarities between the geometric version of the axioms of $G$-TCF (the condition $(\clubsuit^{\mathrm{g}})$ in item $(2)$ above) and the geometric axiomatizations for the theories of fields with Hasse-Schmidt derivations (see \cite{K3} and \cite{HK}). For any $K$-scheme $T$, the scheme $T^{\sigma_1}\times \ldots \times T^{\sigma_e}$ corresponds to $\nabla T$ which is the \emph{prolongation} of $T$ (see Section 5.1 in \cite{HK}) and the geometric axioms of $G$-TCF can be put into an equivalent form, which would resemble the geometric axioms from \cite{HK}. 
 \end{enumerate}
\end{remark}
We proceed now to show that the theory $G-\tcf$ axiomatises the class of existentially closed $G$-transformal fields.

\begin{lemma}\label{embeds}
 Every $G$-transformal field embeds in a model of $G-\tcf$.
\end{lemma}

\begin{proof}
We take a $G$-transformal field $(K,\overline{\sigma})$ and recursively construct a tower of $G$-transformal fields
$$(K,\overline{\sigma})=:(L^0,\overline{\sigma}^0)\subseteq (L^1,\overline{\sigma}^1)\subseteq\ldots$$
in such a way that $L^{i+1}$ 
satisfies ($\clubsuit$) for all $I,J$ \emph{defined} over $L^i$ (i.e. $I,J$ have generators with coefficients from $L^i$).
Assume that we already have $(L^0,\overline{\sigma}^0),\ldots, (L^i,\overline{\sigma}^i)$ and let $(n_{\alpha},I_{\alpha},J_{\alpha})_{0<\alpha<\kappa}$ be an enumeration of all
$$n_{\alpha}\in\mathbb{N}_{>0},\;\;I_{\alpha},J_{\alpha}\trianglelefteqslant L^i[X_1,\ldots,X_e],\;\;|X_1|=\ldots=|X_e|=n_{\alpha}$$
as in ($\clubsuit$) from the axioms of $G-\tcf$.

We recursively define an auxiliary tower of $G$-transformal fields $(L^{i,\alpha},\overline{\sigma}^{i,\alpha})$ for $0\leqslant \alpha<\kappa$. We define
$(L^{i,0},\overline{\sigma}^{i,0}):=(L^i,\overline{\sigma}^i)$, and
the limit ordinal step is standard (we take the increasing union of $G$-transformal fields). For the step with a succesor ordinal $\alpha+1$ we proceed as follows.
Assume that $(L^{i,\alpha},\overline{\sigma}^{i,\alpha})$ is given and consider
$$I':=I_{\alpha}L^{i,\alpha}[X_1,\ldots,X_e],\quad J':=J_{\alpha}L^{i,\alpha}[X_1,\ldots,X_e].$$
Since $L^i[X_1,\ldots,X_e]$ and $L^{i,\alpha}$ are linearly disjoint over $L^i$, we have $I'\subsetneq J'$. However, the ideal $I'$ need not be prime, so we consider two cases.
\\
\\
{\bf Case 1} The ideal $I'$ is not prime.
\\
We just take $(L^{i,\alpha+1},\overline{\sigma}^{i,\alpha+1}):=(L^{i,\alpha},\overline{\sigma}^{i,\alpha})$ (so we ``ignore'' this case).
\\
\\
{\bf Case 2} The ideal $I'$ is prime.
\\
We see that $I'$ is $G$-invariant in $L^{i,\alpha}[X_1,\ldots,X_e]$ with the $G$-transformal ring structure defined under Remark \ref{linijka1}. By Remark \ref{linijka1}, there is a $G$-transformal ring structure on $L^{i,\alpha}[X_1,\ldots,X_e]/I'$, which extends $(L^{i,\alpha},\overline{\sigma}^{i,\alpha})$ and which can be extended to the fraction field
of $L^{i,\alpha}[X_1,\ldots,X_e]/I'$. We define $(L^{i,\alpha+1},\overline{\sigma}^{i,\alpha+1})$ as this fraction field with the described $G$-transformal field structure.

For $a:=X_1+I'$ , we have $\sigma_j(a)=X_{j}+I'$,
hence
$f(\sigma_1(a),\ldots,\sigma_e(a))=0$ in $L^{i,\alpha+1}$ for every $f\in I'$,
so we have the following
$$\big(\sigma_1(a),\ldots,\sigma_e(a)\big)\in V_{L^{i,\alpha+1}}(I')=V_{L^{i,\alpha+1}}(I_{\alpha}).$$
On the other hand, we clearly have
$$\big(\sigma_1(a),\ldots,\sigma_e(a)\big)\not\in V_{L^{i,\alpha+1}}(J')=V_{L^{i,\alpha+1}}(J_{\alpha}),$$
since otherwise it will lead to a contradiction with the condition $I'\subsetneq J'$.
\\

We set $(L^{i+1} ,\overline{\sigma}^{i+1}):=\bigcup\limits_{\alpha<\kappa}(L^{i,\alpha},\overline{\sigma}^{i,\alpha})$. The $G$-transformal field we are looking for is defined as
$$(L,\overline{\sigma}^L):=\bigcup\limits_{i<\omega}(L^i,\overline{\sigma}^i).$$

To see that $(L,\overline{\sigma}^L)$ satisfies ($\clubsuit$), take $n\in\mathbb{N}_{>0}$ and $I,J\trianglelefteqslant L[X_1,\ldots,X_e]$ as in the assumptions of ($\clubsuit$). Because $I$ and $J$ are finitely generated, there is $i\in\mathbb{N}$ such that all generators of $I$ and $J$, all polynomials asserting that generators of $I$ are some combinations of generators of $J$, and polynomial asserting that $I\neq J$ lie in $L^i[X_1,\ldots,X_e]$.
For 
$$I^0:=I\cap L^i[X_1,\ldots,X_e],\;\;J^0:=J\cap L^i[X_1,\ldots,X_e],$$
we have
$I^0\subsetneq J^0$ and the ideal $I^0$ is prime and $G$-invariant. Therefore $I^0=I_{\alpha}$, $J^0=J_{\alpha}$ for some ideals from the above procedure (Case 2). So, there is an element $a\in (L^{i,\alpha})^n\subseteq L^n$ such that $\big(\sigma_1(a),\ldots, \sigma_e(a)\big)\in V_{L^{i,\alpha}}(I_{\alpha})\setminus V_{L^{i,\alpha}}(J_{\alpha})$. Obviously
$\big(\sigma_1(a),\ldots, \sigma_e(a)\big)\in V_L(I_{\alpha})=V_L(I)$
and $\big(\sigma_1(a),\ldots, \sigma_e(a)\big)\not\in V_L(J_{\alpha})=V_L(J)$.
\end{proof}

\begin{lemma}\label{mod_complete}
 The theory $G-\tcf$ is model complete.
\end{lemma}

\begin{proof}
We will check that each model of $G-\tcf$ is existentially closed (which amounts to Robinson's Test, see e.g. \cite[Lemma 3.2.7]{tezi}).
Let $(K,\overline{\sigma})\models G-\tcf$ and let $(L,\overline{\sigma})$ be a $G$-transformal field extending $K$.
Assume that for a quantifier free $\mathcal{L}_G$-formula $\varphi(x)$ with parameters from $K$, where $x=(x_1,\ldots,x_n)$,
there is a solution $a$ in $L$. Formula $\varphi(x)$ can be written as
\begin{IEEEeqnarray*}{rCl}
h_1(\sigma_{1}(x),\ldots,\sigma_{e}(x)) 
= \ldots &=& h_m(\sigma_{1}(x),\ldots,\sigma_{e}(x))=0 \\
&\wedge& g(\sigma_{1}(x),\ldots,\sigma_{e}(x))\neq 0,
\end{IEEEeqnarray*}
where $h_1,\ldots,h_m,g\in K[X_1,\ldots,X_{e}]$.

Let $I:=\lbrace f\in K[X_1,\ldots,X_e]\;|\; f(\sigma_1(a),\ldots,\sigma_e(a))=0\rbrace$. If $f\in I$, then for any $j\leqslant e$ we have
\begin{IEEEeqnarray*}{rCl}
\big(\sigma_j(f)\big)\big(\sigma_1(a),\ldots,\sigma_e(a)\big) &=&
f^{\sigma_j}(X_{j\ast 1},\ldots,X_{j\ast e})\big(\sigma_1(a),\ldots,\sigma_e(a)\big)\\
&=& f^{\sigma_j}\big(\sigma_j\sigma_1(a),\ldots,\sigma_j\sigma_e(a)\big) \\
&=& \sigma_j\Big(f\big(\sigma_1(a),\ldots,\sigma_e(a)\big)\Big)=0.
\end{IEEEeqnarray*}
Therefore $\sigma_j(f)\in I$ and $I$ is a $G$-invariant prime ideal.
Now let $J$ be generated as an ideal of $K[X_1,\ldots,X_{e}]$ by $I\cup\lbrace g\rbrace$. 
Because $a\in V_L(I)\setminus V_L(J)$, we have $I\subsetneq J$.
Hence all the assumptions of ($\clubsuit$) from the definition of $G-\tcf$ are satisfied, and therefore $\varphi(x)$ has a solution in $K$.
\end{proof}

\begin{theorem}
The theory $G-\tcf$ is a model companion of the theory of $G$-transformal fields.
\end{theorem}

\begin{proof}
By Lemma \ref{embeds} and Lemma \ref{mod_complete}.
\end{proof}

\begin{remark}
\begin{enumerate}
\item The theory $G-\tcf$ has no quantifier elimination (otherwise 
it would lead to stability, which is impossible by Corollary \ref{unstable}).

\item The reader may compare our axioms of the theory $G-\tcf$ with the axioms of the theory introduced in \cite[Definition 10.]{sjogren}, which share the same models (\cite[Theorem 14.]{sjogren}, Theorem \ref{equiv}).
Those axioms are given in terms of the properties of the absolute Galois groups.

\item Sj\"{o}gren proved in \cite[Theorem 13.]{sjogren} that the theory of existentially closed $G$-transformal fields considered in the language of $G$-transformal fields extended by a certain set of predicates has quantifier elimination.
\end{enumerate}
\end{remark}

\section{Algebraic properties of models of $G-\tcf$}\label{sec3}
In this section, we will give a field-theoretic characterization of existentially closed $G$-transformal fields.
\subsection{Basic properties}
We consider first some obvious properties of existentially closed $G$-transformal fields.
\begin{lemma}\label{perfect}
 For every $G$-transformal field $(K,\overline{\sigma})$, there is a unique $G$-transformal structure on the perfect closure of $K$,
 $K^{\perf}=\bigcup\limits_{n>0}\lbrace a\in K^{\alg}\;|\;a^{p^n}\in K\rbrace$, extending the one on $K$.
\end{lemma}

\begin{proof}
We can assume that $\ch(K)=p>0$.
 Let $K_n:=\lbrace a\in K^{\alg}\;\;|\;\;a^{p^n}\in K\rbrace$, we start with defining a $G$-transformal field structure $\overline{\sigma}^n$ on each $K_n$,
 $$\overline{\sigma}^n=(\sigma_k^n:=\fr^{(-n)}\circ\sigma_k\circ\fr^{(n)})_{k\leqslant e}.$$
 It is enough to show that $\sigma_k^{n+1}\upharpoonright_{K_n}=\sigma_k^n$ for each $k\leqslant e$.
 Assume that $a^{p^n}=b\in K$, then
 \begin{equation*} 
 \sigma_k^{n+1}(a)=\fr^{(-n-1)}(\sigma_k(b^p))=\fr^{(-n)}(\sigma_k(b))=\sigma_k^n(a). \qedhere
  \end{equation*}
\end{proof}
By the lemma above, every existentially closed
$G$-transformal field is perfect.

\begin{lemma}
 If a $G$-transformal field $(K,\overline{\sigma})$ is perfect, then $K^G$
 is perfect as well.
\end{lemma}

\begin{proof}
 If $a^{p^n}=b\in K^G$ for some $n>0$, then for every $k\leqslant e$, we have
 $$\sigma_k(a)^{p^n}=\sigma_k(b)=b=a^{p^n}.$$
 Since the Frobenius map is injective, we conlude that $\sigma_k(a)=a$, so $a\in K^G$.
\end{proof}
We see that the invariants of
every perfect $G$-transformal field $(K,\overline{\sigma})$ are also perfect and therefore 
the field extension $K^G\subseteq K$
is separable of degree $e$, so (by the primitive element theorem) it is a simple extension.

\begin{cor}\label{cor33}
If $(K,\overline{\sigma})\models G-\tcf$, then $K$ and $K^G$ are perfect.
\end{cor}

\begin{lemma}\label{strict}
Any model $(K,\overline{\sigma})$ of $G-\tcf$ is strict.
\end{lemma}

\begin{proof}
We can easily extend the $G$-transformal field structure to $K(X_1,\ldots,X_e)$, which will be strict. Hence, by model completeness, $(K,\overline{\sigma})$ must be strict.
\end{proof}

Let us consider now a property 
which is \emph{not} satisfied by the models of $G-\tcf$. We start with an example.

\begin{example}\label{cpxconj}
We analyze a special $\Zz/2\Zz$-difference field $(\Cc,\sigma)$,
where $\sigma$ is the complex conjugation.
It is not existentially closed since the difference equation
$$x\sigma(x)=-1$$
has no solutions in $(\Cc,\sigma)$, but it has a solution (i.e. $X$) in the following extension
$$(\Cc(X),\sigma),\ \ \ \sigma(X)=-1/X.$$
\end{example}

\begin{theorem}\label{non-sep}
Suppose that $G$ is non-trivial and $(K,\overline{\sigma})\models G-\tcf$. Then $K$ is not separably closed.
\end{theorem}

\begin{proof}
Suppose that $K$ is separably closed and let $F=K^G$. 
After modifying the proof of \cite[Theorem 11.14]{njac}
to the more general case of a separably closed field
(like in \cite[Theorem 4.1]{conrad}), 
we get that $\ch(F)=0$, $F$ is real closed and $K=F(i)$ 
(see also Corollary 9.3 on page 299 in
\cite{lang2002algebra}). Hence $G=\Zz/2\Zz$ and Example \ref{cpxconj} (working also in this case) gives the contradiction.
\end{proof}

\subsection{Algebraic extensions and $G$-closed fields}\label{subsec_alg_ext}
Given a perfect field $C$, its algebraic closure $C^{\alg}$ and a normal closed subgroup $\mathcal{N}\trianglelefteqslant\mathcal{G}:=
\gal(C^{\alg}/C)$ such that $\mathcal{G}/\mathcal{N}\cong G$, we consider the field $K:=(C^{\alg})^{\mathcal{N}}$.
Since $C\subseteq C^{\alg}$ is a Galois extension and $\mathcal{N}$ is normal, we have the following:
\begin{itemize}
\item $\mathcal{N}=\gal(C^{\alg}/K)$,
\item the extension $C\subseteq K$ is Galois,
\item $\gal(K/C)\cong\faktor{\mathcal{G}}{\mathcal{N}}\cong G$ and
\item $C=K^{\gal(K/C)}=K^G$.
\end{itemize} 

We keep the above set-up in this subsection.

\begin{lemma}\label{alg_ext}
The following conditions are equivalent.
\begin{enumerate}
\item There is a non-trivial algebraic field extension $K\subseteq K'$ such that the action of $G$ on $K$ extends to an action of $G$ on $K'$.

\item There is a proper closed subgroup $\mathcal{G}_0<\mathcal{G}$ such that $\mathcal{G}_0\mathcal{N}=\mathcal{G}$.
\end{enumerate}
\end{lemma}
\begin{proof}
$(1)\Rightarrow (2)$ Let $C':=(K')^G$ and $\mathcal{G}_0:=\gal(C^{\alg}/C')$. 
From the assumption, the restriction map $\alpha$
in the commutative diagram below
\begin{equation*}
\xymatrix{  \gal(C^{\alg}/C) \ar[r]^{\pi}  & \gal(K/C)\\
 \gal(C^{\alg}/C') \ar[r]^{\pi'}\ar[u]^{\leqslant} & \gal(K'/C') \ar[u]^{\alpha}
}
\end{equation*}
is an isomorphism. Therefore the composition of the following maps 
$$\mathcal{G}_0\xrightarrow{\subseteq}\mathcal{G}\xrightarrow{\pi}
\gal(K/C)\cong\mathcal{G}/\mathcal{N}$$
is onto. Therefore $\pi(\mathcal{G}_0)=\mathcal{G}_0\mathcal{N}/\mathcal{N}=\mathcal{G}/\mathcal{N}$
and $\mathcal{G}_0\mathcal{N}=\mathcal{G}$. Since the extension $C'\subseteq C^{\alg}$ is Galois, $\mathcal{G}_0$ is a closed subgroup of $\mathcal{G}$.
\\
$(2)\Rightarrow (1)$ Take $C':=(C^{\alg})^{\mathcal{G}_0}$ and $K':=(C^{\alg})^{\mathcal{G}_0\cap \mathcal{N}}$.
Since $K=(C^{\alg})^{\mathcal{N}}$, the fundamental theorem of Galois theory tells us
that $K'=KC'$ (the intersection of closed subgroups corresponds to the compositum of fields) and the restriction map
$$\gal(K'/C')\xrightarrow{\alpha}\gal(K/C)$$
is one-to-one.
From the assumption $\mathcal{G}_0\mathcal{N}=\mathcal{G}$, the map $\alpha$ is also onto, so the result follows.
\end{proof}

\begin{remark}
It is clear that the negation of the condition \ref{alg_ext}(2) means that the natural map $\pi:\mathcal{G}\to G$ (the restriction map $\gal(C^{\alg}/C)\to\gal(K/C)$) is a \emph{Frattini cover}, i.e. if $\mathcal{H}$ is a closed subgroup of $\mathcal{G}$ s.t. $\pi(\mathcal{H})=G$, then $\mathcal{H}=\mathcal{G}$. 
\end{remark}

The following lemma is the field counter part of Lemma 22.5.6 of \cite{FrJa}.
\begin{lemma}\label{acl}
There is a maximal field extension $K\subseteq K'$ (inside $C^{\alg}$) such that the action of $G$ on $K$ extends to $K'$.
\end{lemma}

\begin{proof}
By Zorn's lemma; since there is a natural $G$-action on the union of an increasing chain of fields with $G$-actions.
\end{proof}

\begin{remark}\label{largest}
It is natural to ask whether: 
\begin{enumerate}
\item There is a \emph{largest} underlying field (inside $K^{\alg}$) admitting a $G$-action extending the one on $K$ (such a field would play the role of ``$G$-algebraic closure'' of $K$).

\item Any two difference fields from Lemma \ref{acl} are isomorphic as difference fields.

\item Any two difference fields from Lemma \ref{acl} are isomorphic as pure fields.
\end{enumerate}
Clearly $(1)$ implies $(2)$ and $(2)$ implies $(3)$. In the examples below, we will see that in some situations $(1)$ holds, but in general even $(3)$ need not hold.
\end{remark}

\begin{example}\label{ex39}
Let $q$ be a prime power, $G=\Zz/n\Zz$ and $C=\Ff_q$. We will see that in such a case the question in Remark \ref{largest} has an affirmative answer.

Since $\gal(\Ff_{q^n}/\Ff_q)\cong \Zz/n\Zz$, the field $\Ff_{q^n}$ becomes a $\Zz/n\Zz$-difference field. We have

$$\mathcal{G}=\gal(\Ff_q^{\alg}/\Ff_q)\cong \widehat{\Zz}\cong \prod_{p\in \Pp}\Zz_p,$$
where $\Zz_p$ denotes the additive group of the ring of $p$-adic integers. Let us take the decomposition of $n$ into prime powers

$$n=p_1^{\alpha_1}\ldots p_k^{\alpha_k}.$$
Then we have
$$\mathcal{N}:= p_1^{\alpha_1}\Zz_{p_1}\times\ldots \times p_k^{\alpha_k}\Zz_{p_k} \times \prod_{p\in \Pp\setminus\{p_1,\ldots,p_k\}}\Zz_p.$$
Therefore, there is a \emph{smallest} $\mathcal{G}_0$ ($= \Zz_{p_1}\times\ldots \times\Zz_{p_k}$) such that $\mathcal{G}_0\mathcal{N}=\mathcal{G}$.
Hence the ``$\Zz/n\Zz$-closure'' of $\Ff_{q^n}$ coincides with
$$(\Ff_q^{\alg})^{\mathcal{G}_0\cap \mathcal{N}}=(\Ff_q^{\alg})^{p_1^{\alpha_1}\Zz_{p_1}\times\ldots \times p_k^{\alpha_k}\Zz_{p_k}}.$$
We will see (Corollary \ref{ex39bis}) that the above field with the induced $\Zz/n\Zz$-action is a model of $\Zz/n\Zz-\tcf$.
\end{example}

\begin{remark}
If $(K,\overline{\sigma})\models G-\tcf$ has characteristic $p>0$, then the $G$-difference 
structure on $K\cap \Ff_p^{\alg}$ factors through a $\Zz/n\Zz$-difference structure, where $\Zz/n\Zz$
is a cyclic quotient of $G$, since a finite subgroup of
$\gal(F/\Ff_p)$ is necessarily cyclic for any subfield $F\subseteq \Ff_p^{\alg}$.
\end{remark}

\begin{example}\label{ex311}
Let $\zeta_n\in \Cc$ be the $n$-th primitive root of unity. Then we have
$$\gal(\mathbb{Q}(\zeta_n)/\mathbb{Q})\cong (\Zz/n\Zz)^*.$$
In particular, we have
$$\gal(\mathbb{Q}(\zeta_8)/\mathbb{Q})\cong \Zz/2\Zz\times \Zz/2\Zz,\ \ \ \gal(\mathbb{Q}(\zeta_{16})/\mathbb{Q})\cong \Zz/4\Zz\times \Zz/2\Zz.$$
The tower of fields $\mathbb{Q}\subset \mathbb{Q}(\zeta_8)\subset \mathbb{Q}(\zeta_{16})$ corresponds to the following epimorphism of Galois groups
$$\Zz/4\Zz\times \Zz/2\Zz\to \Zz/2\Zz\times \Zz/2\Zz,\ \ \ (x,y)\mapsto (r_2(x),y),$$
where $r_2$ is the natural epimorphism $\Zz/4\Zz\to\Zz/2\Zz$.
Let $\phi_1\in \gal(\mathbb{Q}(\zeta_8)/\mathbb{Q})$ correspond to $(1,0)$ (so e.g. $\phi_1(\zeta_8)=\zeta_8^3$) and $\phi_2\in \gal(\mathbb{Q}(\zeta_8)/\mathbb{Q})$ correspond to $(0,1)$ (so $\phi_2(\zeta_8)=\zeta_8^{-1}$). Then $(\mathbb{Q}(\zeta_8),\phi_1),(\mathbb{Q}(\zeta_8),\phi_2)$ are $\Zz/2\Zz$-difference fields (extending the non-trivial $\Zz/2\Zz$-difference field structure on $\mathbb{Q}(\zeta_8+\zeta_8^5)$). However, looking at the above epimorphism of Galois groups, we see that $(\mathbb{Q}(\zeta_8),\phi_1)$ does \emph{not} extend to a $\Zz/2\Zz$-difference field structure on $\mathbb{Q}(\zeta_{16})$, but $(\mathbb{Q}(\zeta_8),\phi_2)$ \emph{does} extend using
$$\phi:\mathbb{Q}(\zeta_{16})\to \mathbb{Q}(\zeta_{16}),\ \ \ \phi(\zeta_{16})=\zeta_{16}^{-1}.$$
Hence, in this case, Question $(1)$ in Remark \ref{largest} has a negative answer. It is also easy to see that this example even gives a negative answer to Question $(3)$ in Remark \ref{largest}. If $K=\mathbb{Q}(\zeta_8+\zeta_8^5)$ with the non-trivial $\Zz/2\Zz$-difference structure, then take $K'$ extending $(\mathbb{Q}(\zeta_8),\phi_1)$ and $K''$ extending $(\mathbb{Q}(\zeta_8),\phi_2)$. Then $K'$ can not be isomorphic to $K''$ over $K$, since $K''$ contains $\mathbb{Q}(\zeta_{16})$ being the (unique in $\mathbb{Q}^{\alg}$) splitting field of the polynomial $X^{16}-1$ over $\mathbb{Q}(\zeta_8+\zeta_8^5)$ and $K'$ does not contain $\mathbb{Q}(\zeta_{16})$.
\end{example}

\subsection{Fields of constants and $K$-strongly PAC fields}\label{PACfields}
In this part, we investigate the field of constants
of an existentially closed $G$-transformal field. We show that this field enjoys a \emph{$K$-strongly PAC} property, which we introduce first.
Finally, we obtain Theorem \ref{equiv}, which gives us a nice description of the models of $G-\tcf$ in a pure algebraic terminology.
\begin{definition}
We say that a field $F$ is \emph{$K$-strongly PAC} ($K$-strongly pseudo-algebraically closed), if
$F\subseteq K$ and every $K$-irreducible $F$-variety has an $F$-rational point.
\end{definition}

\begin{remark}
Naturally, being a $K$-strongly PAC field implies being a PAC field for algebraic
subfields of $K$.
\end{remark}

\begin{theorem}\label{thm_kpac1}
If $(K,\overline{\sigma})\models G-\tcf$, then $K^G$ is a $K$-strongly PAC field.
\end{theorem}

\begin{proof}
Let $C:=K^G$ and $V_0\subseteq \Aa^{n}$ be a non-empty $C$-subscheme such that the scheme 
$$V_1:=V_0\times_{\spec(C)}\spec(K)$$
 is $K$-irreducible. Let $V$ be the (multi-)diagonal of $V_1^e$. Then for each $\sigma\in G$, we have $\sigma(V)=V$. Since $V$ is defined over $C$, for each $\sigma\in G$, we have $V^{\sigma}=V$, where $V^{\sigma}$ is as in Remark \ref{rem_axioms}(2). Hence $V$ satisfies the assumptions of $(\clubsuit^{\mathrm{g}})$ in Remark \ref{rem_axioms}(2). Therefore (taking $W=\emptyset$), there is $a\in V_0(K)$ such that $(\sigma_1(a),\ldots,\sigma_e(a))\in V(K)$. Since $V$ is the diagonal of $V_1^e$, we get that $\sigma_1(a)=\ldots=\sigma_e(a)=a$, so $a\in V_0(C)$.

\end{proof}

\begin{cor}\label{K_is_PAC}
If $(K,\overline{\sigma})\models G-\tcf$, then $K$ is a PAC field.
\end{cor}

\begin{proof}
By a result of Ax-Roquette (Corollary 11.2.5 in \cite{FrJa}), any algebraic extension of a PAC field is again PAC. Hence the result follows from Theorem \ref{thm_kpac1}.
\end{proof}

\begin{definition}\label{def318}
 We call a pair $(C,K)$, where $C\subseteq K$ is a field extension,  a \emph{$G$-closed field} if
\begin{itemize}
\item[i)] the field $C$ is perfect,
\item[ii)] the extension $C\subseteq K$ is Galois with the Galois group $G$,
\item[iii)] the extension $C\subseteq K$ satisfies the negations of the equivalent conditions in Lemma \ref{alg_ext}, that is, the restriction map $\gal(C^{\alg}/C)\to \gal(K/C)$ is a Frattini cover (see Definition \ref{frattini.def}).
\end{itemize}
\end{definition}

\begin{remark}
The first two items of the above definition lead to the set-up of
Subsection \ref{subsec_alg_ext}, so to the assumptions of Lemma \ref{alg_ext}, needed
in the last item.
\end{remark}

\begin{remark}
If $(K,\overline{\sigma})\models G-\tcf$, then $(K^G,K)$ is a $G$-closed field.
\end{remark}

\begin{proof}
Proof. By Corollary \ref{cor33}, the field $K^G$ is perfect. Clearly, the extension $K^G\subseteq K$ is Galois with the Galois group $G$. Since the existence of an algebraic extension as in Lemma \ref{alg_ext}.(1) is an elementary property (for a fixed degree of the extension) and $(K,\bar{\sigma})$ is existentially closed, the last item from Definition \ref{def318} holds as well.
\end{proof}

We recall the definition of a bounded field.

\begin{definition}
A field $F$ is \emph{bounded} if for every natural number $n>0$, there are only finitely many extensions of $F$ of degree $n$ in $F^{\alg}$.
\end{definition}

\begin{theorem}\label{C_is_bounded}
 If $(C,K)$ is a $G$-closed field, then $C$ is a bounded field.
\end{theorem}

\begin{proof}
 After a standard Galois theory argument, we conclude that
 $\mathcal{N}:=\gal(C^{\alg}/K)$ and $\mathcal{G}:=\gal(C^{\alg}/C)$ fit into the assumptions of Lemma \ref{alg_ext}.
 Let $g_1,\ldots,g_{e}$ denote the representatives of different cosets of $\mathcal{G}/\mathcal{N}$, and let $\mathcal{G}_0$
 be a topological subgroup of $\mathcal{G}$ generated by $g_1,\ldots,g_{e}$. We have $\mathcal{G}_0\mathcal{N}=\mathcal{G}$,
 so $\mathcal{G}$ is finitely generated as a profinite group.
 Alternatively, since $\mathcal{G}\to\gal(K/C)$ is a Frattini cover,
$$\rk(\mathcal{G})=\rk(\gal(K/C))<\infty.$$
 By \cite[Proposition 2.5.1.a)]{ribzal}, for each $n>1$, the number
 of closed subgroups of $\mathcal{G}$ of index $n$ is finite.
\end{proof}

\begin{cor}\label{4_to_2}
If $(K,\overline{\sigma})\models G-\tcf$, then $K^G$ is a bounded $K$-strongly PAC field.
\end{cor}
Before stating the next theorem, we recall one more definition which will be needed in its proof and also in the sequel. 

\begin{definition}
 A profinite group $G$ is \emph{projective} if for every epimorphism of profinite groups $\pi: B\to A$ and homomorphism of profinite groups $\phi_A:G\to A$, there is a homomorphism of profinite groups $\phi_B:G\to B$ such that
 $$\pi\circ\phi_B=\phi_A.$$
 \end{definition}

\begin{theorem}\label{elem}
Suppose that $(C,K)\subseteq (C',K')$ is an extension of $G$-closed
 fields such that $C$ and $C'$ are PAC fields. Then the field extension $C\subseteq C'$ is elementary.
\end{theorem}
\begin{proof}
By Corollary 20.3.4 in \cite{FrJa}, it is enough to show that the restriction map
 $r:\mathcal{G}(C')\to \mathcal{G}(C)$ 
is an isomorphism, where we use the notation $\mathcal{G}(F):=\gal(F^{\alg}/F)$ for a field $F$. Since the extension $C\subseteq C'$ is regular
regular (e.g. by using the fact that the restriction map 
$\mathcal{G}\to\gal(K/C)$
is a Frattini cover), this map is an epimorphism. We have the following commutative diagram:
\begin{equation*}
 \xymatrix{  \mathcal{G}(C') \ar[d]^{}  \ar[rr]^{r} &  & \mathcal{G}(C) \ar[d]^{}   \\
  G \ar[rr]^{\cong} &  &  G.}
\end{equation*}
By \cite[Theorem 11.6.2]{FrJa} (originally proved by Ax), the profinite group $\mathcal{G}(C')$ is projective, 
so (see the beginning of the page 88 in \cite{ChaPil}) there is a closed subgroup 
$\mathcal{G}_0<\mathcal{G}(C')$ 
such that $r|_{\mathcal{G}_0}:\mathcal{G}_0\to\mathcal{G}(C)$ is an isomorphism. From the diagram above, we get that $\mathcal{G}_0 \mathcal{N}=\mathcal{G}(C')$, where
$\mathcal{N}=\ker(\mathcal{G}(C')\to G)$, so $\mathcal{G}_0=\mathcal{G}(C')$ (since $(C',K')$ is $G$-closed field).
\end{proof}

\begin{cor}\label{gcloimpgtcf}
Suppose $(C,K)$ is a $G$-closed field such that $C$ is a PAC field. Then $(K,\gal(K/C))$ is a model of $G-\tcf$.
\end{cor}
\begin{proof}
Consider $(C,K)$ as a $G$-transformal field and let $(C',K')$ be a model of $G-\tcf$ extending $(C,K)$, where $C'=(K')^G$. 
In particular, $(C',K')$ is a $G$-closed field. By Theorem \ref{elem}, the extension $C\subseteq C'$ is elementary. 

By Remark \ref{rem_Cspace}(3), the $G$-transformal field extension $K\subseteq K'$ is elementary, in particular $(K,\gal(K/C))$ is a model of $G-\tcf$.
\end{proof}

\begin{remark}
One could wonder whether for a $G$-closed field $(C,K)$, we get that $C$ is automatically PAC. Example \ref{cpxconj} shows that this is not the case for $G=\Zz/2\Zz$. The referee provided an argument showing that this is not the case for an arbitrary non-trivial (finite) group $G$.
\end{remark}

\begin{theorem}\label{kpacimpgcl}
Assume that $C$ is perfect and $C\subseteq K$ is a Galois extension with 
$\gal(K/C)\cong G$.
If $C$ is $K$-strongly PAC, then $(C,K)$ is a $G$-closed field.
\end{theorem}
\begin{proof}
Suppose $(C,K)$ is not a $G$-closed field. Then there is a non-trivial algebraic $G$-transformal field extension $(C,K)\subset (C',K')$. Take $a\in C'\setminus C$ and let $f\in K[X]$ be the minimal polynomial of $a$ over $K$. Then, for any $g \in G$ we have
$$0=\sigma_g(f(a))=\sigma_g(f)(\sigma_g(a))=\sigma_g(f)(a).$$
Hence $\sigma_g(f)$ is a minimal polynomial of $a$ over $K$ as well, so $\sigma_g(f)=f$. Therefore $f\in C[X]$, $f$ is irreducible over $K$ and $\deg(f)>1$. Hence $C$ is not $K$-strongly PAC (and even does not satisfy $(3)$ in the theorem below).
\end{proof}

\begin{theorem}\label{equiv}
Suppose that $C\subseteq K$ is a finite Galois extension of perfect fields such that $\gal(K/C)\cong G$. The following are equivalent.
\begin{enumerate}
\item The $G$-tranformal field $(K,\gal(K/C))$ is a model of $G-\tcf$.
\item The field $C$ is $K$-strongly PAC.
\item The field $C$ is PAC and if $f\in C[X]$ is irreducible over $K$, then $\deg(f)=1$.
\item The field $C$ is PAC and $(C,K)$ is a $G$-closed field.
\end{enumerate}
\end{theorem}
\begin{proof}
By Corollary \ref{4_to_2}, $(1)$ implies $(2)$. Clearly $(2)$ implies $(3)$. By the proof of Theorem \ref{kpacimpgcl}, $(3)$ implies $(4)$. By Corollary \ref{gcloimpgtcf}, $(4)$ implies $(1)$.
\end{proof}

\begin{remark}
The referee has provided a proof of the following result: if $(C,K)$ is a $G$-closed field and $V$ is a $C$-irreducible variety, then $V\times_{\spec(C)}\spec(K)$ is a $K$-irreducible variety. In particular, if in addition $C$ is PAC, then any $V$ as above is absolutely irreducible. This gives a purely algebraic proof of the implication (4)$\Rightarrow$(2) in Theorem \ref{equiv} above.
\end{remark}

\begin{cor}\label{ex39bis}
The $\mathbb{Z}/n\mathbb{Z}$-transformal field mentioned at the end of Example \ref{ex39} is a model of $\mathbb{Z}/n\mathbb{Z}-\tcf$.
\end{cor}

\begin{proof}
By \cite[Corollary 11.2.4]{FrJa} any infinite algebraic extension of a finite field is a PAC field. Obviously, the field extension $\mathbb{F}_q\subseteq (\mathbb{F}_q^{\alg})^{\mathcal{G}_0\cap\mathcal{N}}$ is infinite. Because $\mathcal{G}_0$ is the smallest subgroup satisfying $\mathcal{G}_0\mathcal{N}=\mathcal{G}$, the considered field is $G$-closed, hence it satisfies item (4) of Theorem \ref{equiv}.
\end{proof}

\begin{remark}
The following seems to be an interesting question: which $G$-fields $(K,\bar{\sigma})$ have an \emph{algebraic} (in the field sense) extension to a model of $G-\tcf$?
\begin{enumerate}
\item Corollary \ref{ex39bis} above shows that the $\Zz/n\Zz$-field $(\Ff_{p^n},\Fr)$ has an algebraic (in the field sense) extension to a model of $\Zz/n\Zz-\tcf$.

\item The referee has pointed out to us an argument showing that if $K$ is countable Hilbertian (see Chapter 12 in \cite{FrJa}), then any $G$-field $(K,\bar{\sigma})$ has an algebraic extension to a model of $G-\tcf$. This case is in a way orthogonal to the one considered in the previous item, since finite fields are ``very'' non-Hilbertian.

\item Example \ref{cpxconj} shows that the $\Zz/2\Zz$-field $(\Cc,\sigma)$, where $\sigma$ is the complex conjugation, does not have an algebraic extension to a model of $G-\tcf$.
\end{enumerate}
\end{remark}

\begin{remark}
Using our results, we can quite easily create ``geometric representability'' situations as in \cite[Def. 1.2(1)]{bech}. If $(K,\overline{\sigma})$ is a model of $G-\tcf$, then we put a natural $G$-structure (as in the paragraph after Remark \ref{linijka1}) on $K(\bar{X})$ (where $|\bar{X}|=e=|G|$), and we obtain a $G$-transformal field extension. By Lemma \ref{embeds}, the $G$-difference field $K(\bar{X})$ has a $G$-difference extension to a model $(M,\overline{\sigma})$ of $G-\tcf$. By Lemma \ref{mod_complete}, the $G$-transformal extension $(K,\overline{\sigma})\subset (M,\overline{\sigma})$ is elementary. Hence the extension of bounded PAC fields (by Corollary \ref{4_to_2}) $K\subset M$ is elementary as well. Using Galois theory, we can find an intermediate $G$-transformal field $K\subseteq L\subseteq K(\bar{X})$ such that $\gal(K(\bar{X})/L)\cong G$. By \cite[Theorem 1.8]{bech}, the group of \emph{$\theo(M)$-elementary} (i.e. preserving the set of all $K(\bar{X})$-sentences which are true in $M$)  automorphisms of $K(\bar{X})$ over $L$ is abelian. Therefore, we can easily construct examples when the usual Galois group is very different from the group of elementary automorphisms.
\end{remark}

\subsection{Absolute Galois groups and Frattini covers}\label{secufc}
In this subsection, we will give a description of the absolute Galois group of a model of $G-\tcf$ using Sj\"{o}gren's results from \cite{sjogren}.

The following is a special case of \cite[Proposition 22.10.5]{FrJa}, but we provide a quick proof.

\begin{lemma}\label{nowy1}
Assume that $C\subseteq K$ is a finite Galois extension of degree $e$ and that $C$ is $K$-strongly PAC. Then $C$  has no finite extensions of degree relatively prime to $e$.
\end{lemma}

\begin{proof}
It is enough to notice (using the $K$-strongly PAC assumption) that if $f\in C[X]$ is irreducible of degree relatively prime to $e$, then it remains irreducible over $K$.
\end{proof}

We recall one more definition.

\begin{definition}
A field $F$ is \emph{quasi-finite}, if it is a perfect field such that
for each natural number $n>0$ there exists
a unique extension $F\subseteq F_n$ (in a fixed field $F^{\alg}$) of degree $n$ and
$$F^{\alg}=\bigcup\limits_{n>0}F_n.$$
\end{definition}

By a theorem of Ax (see \cite{ax1}), pseudo-finite fields can be characterized as PAC quasi-finite fields.

\begin{cor}
Let $(K,\overline{\sigma})$ be a model of $G-\tcf$ and $C=K^G$. Then neither $K$ nor $C$ is quasi-finite (hence they are also not pseudo-finite).
\end{cor}

\begin{proof}
It follows from Theorem \ref{equiv} and Lemma \ref{nowy1}.
\end{proof}

We recall below the definition of a Frattini cover.

\begin{definition}\label{frattini.def}
Let $H,N$ be profinite groups and $\pi: H\to N$ be a continuous homomorphism.
The mapping $\pi$ is called \emph{Frattini cover} if for each closed subgroup $H_0$ of $H$, the condition $\pi(H_0)=N$ implies that $H_0=H$.
\end{definition}

The notions of a projective profinite group and a Frattini cover are coming together (similarly as in Theorem \ref{elem}) in the next classical result.

\begin{theorem}\label{thmuft}\cite[Prop. 22.6.1]{FrJa}
Each profinite group $H$ has a Frattini cover $\widetilde{H}\to H$ which is unique up to isomorphism (and called the \emph{universal Frattini cover}) and satisfying
the following equivalent conditions.
\begin{enumerate}
\item The map $\widetilde{H}\to H$ is a projective Frattini cover of $H$.
\item The map $\widetilde{H}\to H$ is the largest Frattini cover of $H$.
\item The map $\widetilde{H}\to H$ is the smallest projective cover of $H$.
\end{enumerate}
\end{theorem}

For any profinite group $H$, we denote by $\tilde{H}\to H$ its universal Frattini cover.

\begin{example}
It is easy to see that the projection map $\Zz_p\to \Zz/p\Zz$ is the universal Frattini cover. From Example \ref{ex39}, we know that $\Zz_p$ is also isomorphic to the absolute Galois group of the constant field of a model of $\Zz/p\Zz-\tcf$. We will see below a natural generalization of this observation.
\end{example}

We quote now the results of Sj\"{o}gren regarding the absolute Galois groups of models of $G-\tcf$. We provide a quick proof which uses our previous results (however, we point out to Remark \ref{remsjo}).

\begin{theorem}\label{sjor}\cite[Theorems 5. and 6.]{sjogren}
Suppose $(K,\bar{\sigma})\models G-\tcf$ and $C=K^{\bar{\sigma}}$. Then we have the following.
\begin{enumerate}
\item The natural map $\gal(C)\to G$ is the universal Frattini cover of $G$.
\item There is an isomorphism
$$\gal(K)\cong \ker(\widetilde{G}\to G).$$
\end{enumerate}
\end{theorem}

\begin{proof}
Clearly, it is enough to show $(1)$. Since $(K,\bar{\sigma})$ is $G$-closed, by Lemma \ref{alg_ext} the map $\gal(C)\to G$ is a Frattini cover. By Theorem \ref{equiv}, $C$ is PAC, so $\gal(C)$ is a projective profinite group. By Theorem \ref{thmuft}, the map $\gal(C)\to G$ is the universal Frattini cover of $G$.
\end{proof}

\begin{remark}\label{remsjo}
Sj\"{o}gren's results \cite[Theorems 5. and 6.]{sjogren} actually apply to a much more general 
case of existentially closed $G$-transformal fields
for an arbitrary group $G$. We will comment more on this case in Section 5.
\end{remark}

\section{Model-theoretic properties of $G-\tcf$}\label{logic}
In this section, we use the results of Section \ref{sec3} to determine the model-theoretic properties of existentially closed $G$-transformal fields. 
\subsection{Model complete difference fields}
In this part, we prove several general results about the theories $G-\tcf$,
in particular we describe:  the types, the model-theoretic algebraic closure and the completions of $G-\tcf$.
These results and their proofs are basically the same as in \cite{acfa1}. Therefore, instead of copying them, we decided to isolate a general set-up for both $G-\tcf$ and ACFA, in which the arguments remain the same.

Assume that $\mathcal{L}$ is the language of rings extended by the set of unary function symbols $\overline{\sigma}=(\sigma_i)_{i\in I}$. Let $T$ denote an $\mathcal{L}$-theory, which contains axioms asserting that models of $T$ are fields equipped with a set of automorhisms corresponding to $(\sigma_i)_{i\in I}$.

Assume that 
if
\begin{itemize}
\item the $\mathcal{L}$-structures $(K_1,\bar{\sigma}_1)$, $(K_2,\bar{\sigma}_2)$ are models of the theory $T$,
\item the field extensions $E\subseteq K_1$, $E\subseteq K_2$ are regular,
\item we have $\sigma_{1,i}\upharpoonright_E=\sigma_{2,i}\upharpoonright_E$ for each $i\in I$,
\end{itemize}
then the unique difference structure $\bar{\tau}$ on $(K_1\otimes_EK_2)_0$
(by Corollary at \cite[\S17, A.V.140]{baki2} $K_1\otimes_EK_2$ is a domain),
which extends $\overline{\sigma}_1$ and $\overline{\sigma}_2$,
satisfies:
\begin{itemize}
\item $\big((K_1\otimes_EK_2)_0,\bar{\tau}\big)\models T$.
\end{itemize}
Assume moreover that $T$ has a model companion $T'$, whose models are perfect fields.

\begin{remark}\label{tensor1}
We easily see that $G-\tcf$ and ACFA (taken for $T'$) fit to our set-up.
\begin{enumerate}
\item Let $(K_1,\overline{\sigma}_1),(K_2,\overline{\sigma}_2)$ be $G$-transformal fields and suppose that there exists $E\subseteq K_1\cap K_2$ 
such that $K_1$ and $K_2$ are regular over $E$, and $\sigma_{1,i}\upharpoonright_E=\sigma_{2,i}\upharpoonright_E$ for each $i\leqslant e$.
Then the unique difference structure $\bar{\tau}$ on $(K_1\otimes_EK_2)_0$,
which extends $\overline{\sigma}_1$ and $\overline{\sigma}_2$,
is a $G$-transformal field.

\item Similarly for the standard theory of difference fields, i.e. with $G=\mathbb{Z}$.
\end{enumerate}
\end{remark}

\begin{prop}\label{acl}
 Suppose that $(K,\overline{\sigma})\models T'$ and $A\subseteq K$.
 The model-theoretic algebraic closure of $A$, $\acl_{\overline{\sigma}}(A)$,
 coincides with $(\langle A\rangle_{\overline{\sigma}})^{\acl}(K)$, 
 where $\langle A\rangle_{\overline{\sigma}}$ is the smallest difference field containing $A$ 
 and $L^{\acl}(K)$ denotes the relative algebraic closure of $L$ (a subfield of $K$) in $K$.
\end{prop}

\begin{proof}
 Similarly to the proof of \cite[Proposition 1.7]{acfa1},
 there are infinitely many realisations of $\tp(a/A)$ in a sufficiently saturated model for any $a\not\in E:=(\langle A\rangle_{\overline{\sigma}})^{\acl}(K)$.
 To see this, we note that $E$ is perfect, 
 so the extension $E\subseteq K$ is regular. Then in $(K\otimes_E K)_0$ we have two distinct realisations of $p$: $a\otimes 1$, $1\otimes a$, and we can continue this process.
\end{proof}

\begin{prop}\label{equivalent}
 Let $(K_1,\overline{\sigma}_1),(K_2,\overline{\sigma}_2)\models T'$ and suppose that there exists $E\subseteq K_1\cap K_2$ algebraically closed
 (in the sense of model theory) in both $K_1$ and $K_2$, such that $\sigma_{1,k}\upharpoonright_E=\sigma_{2,k}\upharpoonright_E$ for each $k<e$.
 Then
 $$(K_1,\overline{\sigma}_1)\equiv_E(K_2,\overline{\sigma}_2).$$
\end{prop}

\begin{proof}
 We embed $\big((K_1\otimes_EK_2)_0,\bar{\tau}\big)$ into a model of $T'$, say $M$.
 The model completeness of $T'$ implies that $K_1,K_2\prec M$.
\end{proof}

\begin{cor}
 Let $(K_1,\overline{\sigma}_1),(K_2,\overline{\sigma}_2)\models T'$ have a common prime field $P$.
 For $E_i:=\acl_{\overline{\sigma}}(P)$ considered in $K_i$, $i=1,2$, we have
 $$(E_1,\overline{\sigma}_1\upharpoonright_{E_1})\cong(E_2,\overline{\sigma}_2\upharpoonright_{E_2})
 \;\;\iff\;\; (K_1,\overline{\sigma}_1)\equiv(K_2,\overline{\sigma}_2).$$
\end{cor}

\begin{proof}
 To prove the implication from left to right we use Proposition \ref{equivalent} after identifying $E_1$ with $E_2$.
 To prove the opposite implication, we embed $K_1$ and $K_2$ in some monster model of $T'$, say $M$, and observe that $\acl_{\overline{\sigma}}(P)$
 considered in $M$ is equal to both $E_1$ and $E_2$.
\end{proof}

\begin{cor}\label{typ_iso}
 Let $(K,\overline{\sigma})\models T'$, $E$ be a substructure of $K$ and let $a$, $b$ be finite tuples from $K$.
 Then
 $$\tp(a/E)=\tp(b/E)$$
 if and only if there exists an $E$-isomorphism betwen
 $\acl_{\overline{\sigma}}(Ea)$ and $\acl_{\overline{\sigma}}(Eb)$, sending $a$ to $b$.
\end{cor}

\begin{proof}
 If $\tp(a/E)=\tp(b/E)$, then we embed $K$ in a monster model $M\models T'$, so there is $f\in\aut(M/E)$ sending $a$ to $b$.
 Now, suppose that $f:\acl_{\overline{\sigma}}(Ea)\to\acl_{\overline{\sigma}}(Eb)$ is an isomorphism over $E$ such that $f(a)=b$.
 We consider $K\otimes_{\acl_{\overline{\sigma}}(Ea)} K$, where the $\acl_{\overline{\sigma}}(Ea)$-algebra structure 
 on the right-hand side of the tensor product
 is given by $f$. 
 To finish the proof, observe that $x\otimes 1\mapsto 1\otimes x$ is an elementary isomorphism (over $E$)
 between two copies of $K$ sending $b$ to $a$, hence it extends to an automorphism of a monster model.
\end{proof}

\subsection{Supersimplicity of $G-\tcf$}
In this part, we show some finer results about model theory of existentially closed $G$-transformal fields and the corresponding $K$-strongly PAC fields of constants. The results about $G-\tcf$ correspond to the known results about ACFA and the results about $K$-strongly PAC fields correspond to the known results about PAC fields with one important difference: we obtain model completeness and elimination of imagineries for $K$-strongly PAC fields only after adding \emph{finitely} many constants.

We recall a well-known theorem about the model-theoretic properties of PAC fields.
 
\begin{theorem}\label{pac_simple}
 \cite[Fact 2.6.7]{kim1} Let $F$ be a PAC field. We have the following:
 \begin{itemize}
  \item[i)] The theory $\theo(F)$ is simple if and only if the field $F$ is bounded;
  \item[ii)] The theory $\theo(F)$ is supersimple if and only if the field $F$ is bounded and perfect;
  \item[iii)] The theory $\theo(F)$ is stable if and only if the field $F$ is separably closed.
 \end{itemize}
\end{theorem}

Let us fix a sufficiently saturated model $(K,\overline{\sigma})$ of $G-\tcf$ and let $C=K^G$. 
In this subsection, $\bar{c}$ denotes both the tuples of elements $\bar{c}$ and $\bar{d}$ discussed in Remark \ref{rem_Cspace}.
We directly obtain the following.

\begin{cor}\label{unstable}
 Any completion of $G-\tcf$ is unstable. For any $(K,\overline{\sigma})\models G-\tcf$, the theory $\theo(K^G)$ (in the language of rings) is supersimple.
\end{cor}

\begin{proof}
The first part is a consequence of Corollary \ref{K_is_PAC}. 
The second part is a consequence of Theorem \ref{thm_kpac1} and Theorem \ref{C_is_bounded}.
\end{proof}

Similarly as in the case of ACFA, we can compute the SU-rank of the field of constants.

\begin{prop}\label{su_1}
The theory $\theo(C,\bar{c})$ is supersimple of SU-rank $1$.
\end{prop}

\begin{proof}
We prove that SU-rank of $\theo(C)$ is equal to $1$, which implies the thesis.
It is enough to show that for a type of an element of $C$, any of its forking extension is algebraic. It is clear from the description of forking in PAC fields (as the field-theoretic algebraic closure in the case of perfect PAC fields) given in 4.7 and 4.8 of \cite{ChaPil}.
\end{proof}

\begin{prop}\label{su_su}
The theory $\theo(K,\bar{\sigma})$ is supersimple of SU-rank e.
\end{prop}
\begin{proof}
It follows from Proposition \ref{su_1}, Lascar inequality and the bi-interpretability of the theories $\theo(K,\bar{\sigma})$ and $\theo(C)$ (after adding appropriate constants) given by Remark \ref{rem_Cspace}. We give a more detailed argument below.

By Proposition \ref{su_1} and Lascar's inequality we have
$$\su_{\theo(C,\bar{c})}(C^e)=e.$$
Let $\mathcal{C}$ 
be the structure with the universe $C^e$ and the full induced first-order structure coming from the language of rings extended by the constants $\bar{c}$. Then we still have
$$\su_{\theo(\mathcal{C})}(C^e)=e.$$
By Remark \ref{rem_Cspace}, after adding finitely many extra constants, 
the structure $\mathcal{C}$ is inter-definable with a $G$-transformal field which is isomorphic to $(K,\bar{\sigma})$. Therefore the theory of $(K,\bar{\sigma})$ is supersimple and we have
\begin{equation*}
\su(\theo(K,\bar{\sigma}))=e.\qedhere
\end{equation*}
\end{proof}

Let $F$ be a bounded PAC field. It is proved in Section 4.6 of \cite{ChaPil} that the theory of $F$, in the language of fields with countably many extra constants coding all the finite extensions of $F$, is model complete and eliminates imaginaries.

\begin{theorem}\label{complete411}
Suppose that $(C,K)$ is a $K$-strongly PAC field and let $\bar{c}$ be the finite tuple of constants from Remark \ref{rem_Cspace}. Then the theory $\theo(C,+,\cdot,\bar{c})$ is model complete and eliminates imaginaries.
\end{theorem}
\begin{proof} By Remark \ref{rem_Cspace}, the theory $\theo(C,+,\cdot,\bar{c})$ is bi-interpretable with the theory $\theo(K,+,\cdot,\bar{\sigma},\bar{c})$. By Lemma \ref{mod_complete}, the theory $\theo(K,+,\cdot,\bar{\sigma},\bar{c})$ is model complete. Hence the theory $\theo(C,+,\cdot,\bar{c})$ is model complete as well.

To show the elimination of  imaginaries, we repeat a standard reasoning.
The proof originating from Proposition 3.2 in \cite{hru-manu} (which was also used in \cite[1.10]{acfa1} and \cite[2.9]{ChaPil}) goes through in our context. 
We just need to use the transcendence dimension instead of the pair (transformal dimension,transformal degree) used in \cite{acfa1} or the fundamental order used in \cite{ChaPil}.
\end{proof}

\begin{remark}
As in Section 1.13 in \cite{acfa1}, the theory $\theo(C)$ does \emph{not} eliminate imagineries.
\end{remark}

\section{Fields with operators}\label{last}
In this section, we discuss several possible generalizations of the transformal fields which have been considered in this paper.
Let us see first that 
the notion of a $G$-transformal field fits into the formalism of iterative $\mathcal{D}$-operators from \cite{moosca1}. For each $n\in \Nn_{>0}$ and a $k$-algebra $R$, we define $\mathcal{D}_n(R)=R^e$, and for each $m\geqslant n$, we define $\pi_{m,n}=\id$. Then we obtain a \emph{generalized Hasse-Schmidt system} which is in a way constant. We put an \emph{iterative} structure on our system using the comultiplication map coming from the the Hopf algebra of functions
$$\Func(G,R)\cong_R \Func(G,k)\otimes_kR.$$
For such a choice of an iterative generalized Hasse-Schmidt system $\mathcal{D}$, it is easy to check that a $\mathcal{D}$-ring structure on a $k$-algebra $R$ corresponds exactly to an action of $G$ on $R$ by $k$-algebra automorphisms.

The above observation can be generalized by replacing the finite dimensional Hopf algebra $\Func(G,k)$ with an arbitrary finite dimensional Hopf algebra $\mathcal{H}$. In such a case, the iterative $\mathcal{D}$-operators on $R$ correspond exactly to the group scheme actions of $\spec(\mathcal{H})$ on $\spec(R)$. In \cite{HK}, the model theory of such group scheme actions was analyzed for a \emph{local} $\mathcal{H}$ (corresponding to an \emph{infinitesimal} group scheme). These actions correspond to certain \emph{iterative truncated Hasse-Schmidt derivations}. It is worth mentioning that some other types of iterative generalized Hasse-Schmidt systems (e.g. the one corresponding to standard iterative Hasse-Schmidt derivations) are not of this form, i.e. they are not ``constant''. But they can be approximated by limits of constant systems coming from finite group schemes. In the case of the standard iterativity, it corresponds exactly to the fact that the formalization of the additive group is the direct limit of its Frobenius kernels. We expect that this approximation phenomena is closely related to the companionability of the theories of fields with these types of iterative operators (see Section 6 in \cite{HK}).

As it was mentioned in the introduction, any finite group scheme $\mathfrak{g}$ fits into an exact sequence
$$1\to \mathfrak{g}^0\to \mathfrak{g}\to \mathfrak{g}_0\to 1,$$
 where $\mathfrak{g}^0$ is infinitesimal and $\mathfrak{g}_0$ is \'{e}tale.  Over a separably closed field, an \'{e}tale finite group scheme may be identified with a finite group $G$ and over an arbitrary ground field, it corresponds to an action of the absolute Galois group on a finite group. This paper deals with the model theory of the actions of the pure finite group $G$.

\begin{remark}
It looks like the case of an arbitrary \'{e}tale finite group scheme can not be easily reduced to the case of a constant group scheme. One could hope that if $\mathfrak{g}_0$ is an \'{e}tale finite group scheme such that for a finite Galois extension $k\subseteq l$, the $l$-group scheme $\mathfrak{g}_0\otimes_kl$ is constant and corresponds to a finite group $G$, then the existentially closed $\mathfrak{g}_0$-group scheme actions may be understood using the models of $G-\tcf$. Unfortunately, this idea does not work in a direct way. Take for example
$$\mathfrak{g}_0:=\mu_{16,\mathbb{Q}}\times \Zz/2\Zz$$
and a $\mathfrak{g}_0$-action on $\spec(\mathbb{Q}(\zeta_8))$ which is trivial on $\mu_{16,\mathbb{Q}}$ and which is given on $\Zz/2\Zz$ by $\phi_1$ as in Example \ref{ex311}.
By Example \ref{ex311}, it is impossible to extend this $\mathfrak{g}_0$-action to any field containing $\zeta_{16}$, and only over such fields the group scheme $\mu_{16,\mathbb{Q}}$ becomes a constant group scheme. We will deal with the model theory of \'{e}tale finite group scheme actions (and more generally, any finite group scheme actions) in a forthcoming paper.
\end{remark}

One could also consider the model theory of $G$-transformal fields for an arbitrary (possibly infinite) $G$. However, the class of existentially closed $G$-transformal fields is often not elementary, as was mentioned in the introduction. Sj\"{o}gren develops some model theory of the class of existentially closed $G$-transformal fields in general, see \cite{sjogren}. We quote below one result from \cite{sjogren}, which we find particularly interesting.
\begin{theorem}[Theorem 6 in \cite{sjogren}]
An existentially closed $G$-transformal fields is algebraically closed if and only if the profinite completion of $G$ (denoted $\widehat{G}$) is a projective group.
\end{theorem} 
\begin{remark} 
Note that the following profinite groups:
\begin{itemize}
\item $\widehat{\mathbb{Z}}$ (corresponding to ACFA);
\item$\widehat{F_n}$ (corresponding to ACFA$_n$);
\item $\widehat{\mathbb{Q}}=0$ (corresponding to $\mathbb{Q}$ACFA)
\end{itemize}
are projective. On the other hand $\hat{\Zz}\times\hat{\Zz}$ (corresponding to fields with two commuting automorphisms) and finite groups (corresponding to the theory $G-\tcf$) are not.
\end{remark}

\begin{question}\label{last.questions}
We finish this paper with stating several questions.
\begin{enumerate}
\item Is there an algebraic (or first-order) description of perfect fields $C$ which are $\gal(K/C)$-closed for a finite Galois extension $C\subseteq K$?
\\
\\
Obviously such fields can be characterized by the existence of an open subgroup $\mathcal{N}\leqslant \gal(C^{\alg}/C)$ such that the quotient map
$$\gal(C^{\alg}/C)\to \gal(C^{\alg}/C)/\mathcal{N}$$
is the universal Frattini cover (in particular, the profinite group $\gal(C^{\alg}/C)$ is projective).

\item Let $G$ be an arbitrary group. Does $G-\tcf$ exist if and only if $\Zz\times \Zz$ does not embed into $G$?
\\
\\
It is hinted in \cite{FrJa} that the left-to-right implication in Question \ref{last.questions}.(2) holds, but we do not know any argument for that.
There may be some hints in the two last sections of \cite{sjogren} that the right-to-left implication does not hold.
\end{enumerate}
\end{question}

\bibliographystyle{plain}
\bibliography{1nacfa}

\end{document}